\documentclass {amsart}

\usepackage {amsmath}
\usepackage {amsthm}
\usepackage {amsfonts}
\usepackage {dirtytalk}
\usepackage {xcolor}
\usepackage[citestyle=alphabetic,bibstyle=alphabetic]{biblatex}

\newcommand {\R} {\mathbb{R}}
\newcommand {\Z} {\mathbb{Z}}

\newcommand {\FF} {\mathcal{F}}

\newcommand {\Fix} {\text{Fix}}

\newtheorem{theorem}{Theorem}
\newtheorem{proposition}{Proposition}

\title {Fixed-Point-Free Involutions on boundary of RACGs}
\author {Aditya De Saha}
\date {}

\bibliography{mybib.bib}

\begin{document}

\begin{abstract}
    About 20 years ago, Bogdan Nica conjectured that the boundary of any
    word-hyperbolic group admits admits a fixed-point-free involution.
    In this very short article, we prove a variation of the conjecture,
    replacing word-hyperbolic groups with Right-angled coxeter groups. This is
    not a solution to Nica's conjecture, but hopefully the techniques will shed
    some light regarding possible approaches to the problem.
\end{abstract}

\maketitle

\section {Introductions, Right-angled Coxeter Groups}

A \emph{Coxeter System} is a pair \( (\Gamma, V) \) where \( \Gamma \) is a
group, and \( V \) is a \emph{finite} set of generators of \( \Gamma \), such
that the group \( \Gamma \) can be presented as
\begin{equation*}
\Gamma = <v \in V : (ab)^{m(a,b)} = 1 \text{ for all } a,b \in V  >
\end{equation*}
Where \( m(a,b) \in \Z \cup \{ \infty \}  \) with the property:
\begin{enumerate}
  \item \( m(a,b) \geq 2 \text{ for all } a,b \in V \)
  \item \( m(a,b) = m(b,a) \) for all \( a,b \in V \)
  \item \( m(a,a) = 2 \text{ for all } a \in V \)
\end{enumerate}
(Interpret the relation \( (ab)^{\infty} = 1 \) as \say{\( a,b \) has no
relation}).

The group \( \Gamma \) is called a \emph{Coxeter} group. In this paper we will
be primarily interested in \emph{Right-angled} coxeter groups (RACGs), for which \(
m(a,b) = 2 \) or \( \infty \) for all \( a,b \in V \).

Davis (\cite{davis_geometry_2008}, \cite{davis_chapter_2001},
\cite{davis_groups_1983}) showed that for any coxeter system \( (\Gamma, V) \)
there exists a \( CAT(0) \) complex \( \Sigma (\Gamma, V) \) on which \( \Gamma
\) acts via reflections. If the context is obvious, we will denote the complex
by \( \Sigma \), ignoring the underlying coxeter system.

We will briefly describe Davis' construction of the complex \( \Sigma \) for the
coxeter system \( ( \Gamma, V) \). Consider the set \( \FF \) of subsets of \( V
\) which generate a finite subgroup of \( \Gamma \). It is a poset under
inclusion. For every maximal element \( F \in \FF \) define \( P_{F}  \) to be a
point. Assuming that \( P_{F'} \) has been defined for every \( F' \supset F \),
define \( P_{F} \) to be the cone on \( \bigcup _{F' \supset F} P_{F'} \). In
the last step, define \( K = P_{\phi} = \text {Cone} \left( \bigcup _{F \neq
\phi } P_{F} \right)  \). Finally, define \( \Sigma = \Gamma \times K / \sim \),
where \( \sim \) is defined as \( (\gamma_{1}, x_{1}) \sim (\gamma_{2}, x_{2})
\) if and only if \( x_{1} = x_{2} \) and \( \gamma_{1}^{-1} \gamma_{2} \in
\left< v \in V \ | \ x_{1} = x_{2} \in P_v \right> \).

\begin{theorem}
    [\cite{davis_geometry_2008}, Theorem 12.2.1 (credited to Gromov)]
    Suppose \( (\Gamma, V) \) is a right-angled coxeter system. Then \( \Sigma
    \) can be given a piecewise Euclidean CAT(0) structure.
\end{theorem}

For a given coxeter system \( (\Gamma, V) \) we define it's \emph{visual
boundary} \( \partial  \Gamma \) to be the visual boundary of its associated
CAT(0) space \( \Sigma(\Gamma, V) \). In other words, it is the set of geodesic
rays in \( \Sigma \) emanating from the unit element \( 1 \in \Gamma \) with the
topology of the uniform convergence on compact sets. We refer the reader to
\cite{dranishnikov_virtual_1997}, \cite{fischer_boundaries_2003} or
\cite{Hosaka2000THEBA} for more details on boundaries of coxeter groups.
Since our group acts on \( \Sigma \) via reflections, it takes geodesics to
geodesics, so the action on \( \Sigma \) gives us an action on the boundary \(
\partial \Gamma \). In other words, the action of \( \Gamma \) on \( \Sigma \)
gives \( \Gamma \) an \( EZ \)-structure. 

Because of the discussion above, we can restrict ourselves to the action of \(
\Gamma \) on \( \Sigma \) instead of \( \partial \Gamma \).


\section {Fixed-Point-Free Involution on the Boundary of RACGs}

Suppose a group \( G \) acts on a space \( X \). We denote the space of fixed
points of \( X \) under the \( G \)-action by \( \Fix (G, X) \).

We prove the following proposition:
\begin{proposition}
  Given a right-angled coxeter group \( \Gamma \), there exists an element \(
  \gamma \in \Gamma \) of order two such that only finitely many points in
  \( \Sigma \) are fixed by \( \gamma \).
\end{proposition}

\begin{proof} [Proof of Proposition.]

  For a subset \( S \subset \Gamma \), we say that \( S \) is \emph{spherical}
  if all elements of \( S \) commute with each other. Pick a spherical subset
  \( S \subset V \) of maximal cardinality. Let \( n = |S| \). Let \( S = \{
  a_1, a_2, \cdots , a_n \}  \), and let \( \Gamma_S \) be the subgroup of
  \( \Gamma \) generated by \( S \). The elements \( a_1, \cdots, a_n \) act
  on \( \Sigma \) by reflections, fixing the subspaces \( \Sigma_{a_1},
  \cdots, \Sigma_{a_n} \) respectively. These hyperplanes intersect with each
  other, since all these elements commute.

  Define our element \( \gamma \) as the product
  \[ \gamma = a_1 a_2 \cdots a_n .\]
  \( \gamma \) is of order two because
  \[ \gamma ^{2} =  (a_1 \cdots a_n) (a_1 \cdots a_n) = a_1^{2} \cdots a_n^{2} 
    = 1 .\]
    By lemma 7.5.1 in \cite{davis_geometry_2008}, since \( S \) is of maximal
    cardinality, \( \Fix (\Gamma_S, \Sigma) \) is a point, let it be \( x_0 \)
    in \( \Sigma \). Again by lemma \( 7.5.4 \) in the same book
    \cite{davis_geometry_2008}, a neighborhood of \( x_0 \) is \( \Gamma_S
    \)-equivariantly a neighborhood of the origin in the canonical
    representation of \( \Gamma_S \) on \( \R^{S} \). But the action of \(
    \gamma = a_1 a_2 \cdots a_n \) corresponds to the antipodal map \( z \to -z
    \) in \( \R^{S} \), so it fixes only the origin \( x_0 \). Since \( \Sigma
    \) is CAT(0) and \( \Gamma \) acts via reflections, the local fixed point \(
    x_0\) is the only global fixed point.  

\end{proof}

Since the action of \( \gamma \) on \( \Sigma \) gives rise to an action of \(
\gamma \) on \( \partial \Sigma \), we obtain as a corollary, the main result of
the paper:

\begin{theorem}
  For any right-angled coxeter group \( \Gamma \), there exists a
  fixed-point-free involution \( \varphi \) on \( \partial \Gamma \).
\end{theorem}

\begin{proof}
  The action of \( \Gamma \) on \( \Sigma \) is via reflections, so it takes
  geodesics to geodesics. Therefore the \( \Gamma \)-action on \( \Sigma \) gives an action
  of \( \Gamma \) on \( \partial \Sigma \). By the proposition, there is an
  element \( \gamma \in \Gamma \) such that \( \gamma \) has only finitely many
  fixed points in \( \Sigma \). Since there are only finitely many fixed points,
  therefore no geodesics are fixed by \( \gamma \), and hence the induced map \(
  \varphi \) on \( \partial \Sigma \) has no fixed points. Also since \( \gamma
  \) has order two, \( \varphi \) is an involution.
\end{proof}

\section{Acknowledgement}

The author would like to thank his advisor Alexander Dranishnikov for many
helpful conversations regarding Coxeter groups and help with theorem 2.

\printbibliography
\end{document}